\documentclass[12pt]{amsart}
\usepackage{amscd}
\usepackage{amsmath}
\usepackage{amssymb}
\usepackage{units}
\usepackage[all]{xy}
\def\frk{\frak}               

\def\Phi{{\frk n}}
\def\Phi{{\frk N}}
\def\opn#1#2{\def#1{\operatorname{#2}}} 
\opn\chara{char} \opn\length{\ell} \opn\pd{pd} \opn\rk{rk}
\opn\projdim{proj\,dim} \opn\injdim{inj\,dim} \opn\rank{rank}
\opn\depth{depth} \opn\sdepth{sdepth} \opn\fdepth{fdepth}
\opn\grade{grade} \opn\height{height} \opn\embdim{emb\,dim}
\opn\codim{codim}  \opn\min{min} \opn\max{max}

\opn\Tr{Tr} \opn\bigrank{big\,rank}
\opn\superheight{superheight}\opn\lcm{lcm}
\opn\trdeg{tr\,deg}
\opn\reg{reg} \opn\lreg{lreg} \opn\ini{in} \opn\lpd{lpd}
\opn\size{size}
\opn\div{div} \opn\Div{Div} \opn\cl{cl} \opn\Cl{Cl}
\opn\Spec{Spec} \opn\Supp{Supp} \opn\supp{supp} \opn\Sing{Sing}
\opn\Ass{Ass} \opn\Min{Min}
\opn\Ann{Ann} \opn\Rad{Rad} \opn\Soc{Soc}
\opn\Im{Im} \opn\Ker{Ker} \opn\Coker{Coker} \opn\Am{Am}
\opn\Hom{Hom} \opn\Tor{Tor} \opn\Ext{Ext} \opn\End{End}
\opn\Aut{Aut} \opn\id{id}  \opn\deg{deg}

\opn\nat{nat}
\opn\pff{pf}
\opn\Pf{Pf} \opn\GL{GL} \opn\SL{SL} \opn\mod{mod} \opn\ord{ord}
\opn\Gin{Gin} \opn\Hilb{Hilb}
\opn\aff{aff} \opn\con{conv} \opn\relint{relint} \opn\st{st}
\opn\lk{lk} \opn\cn{cn} \opn\core{core} \opn\vol{vol}
\opn\link{link} \opn\star{star}
\opn\gr{gr}

\def\pot#1#2{#1[\kern-0.28ex[#2]\kern-0.28ex]}

\opn\dirlim{\underrightarrow{\lim}}
\opn\inivlim{\underleftarrow{\lim}}
\let\to=\rightarrow

\def\Implies{\ifmmode\Longrightarrow \else
        \unskip${}\Longrightarrow{}$\ignorespaces\fi}
\def\implies{\ifmmode\Rightarrow \else
        \unskip${}\Rightarrow{}$\ignorespaces\fi}
\def\iff{\ifmmode\Longleftrightarrow \else
        \unskip${}\Longleftrightarrow{}$\ignorespaces\fi}

\let\:=\colon
\newtheorem{Theorem}{Theorem}[]
\newtheorem{Lemma}[Theorem]{Lemma}
\newtheorem{Corollary}[Theorem]{Corollary}
\newtheorem{Proposition}[Theorem]{Proposition}
\newtheorem{Remark}[Theorem]{Remark}

\newtheorem{Example}[Theorem]{Example}

\let\epsilon\varepsilon
\let\phi=\varphi
\let\kappa=\varkappa
\def\qed{\ifhmode\textqed\fi
      \ifmmode\ifinner\quad\qedsymbol\else\dispqed\fi\fi}
\def\textqed{\unskip\nobreak\penalty50
       \hskip2em\hbox{}\nobreak\hfil\qedsymbol
       \parfillskip=0pt \finalhyphendemerits=0}
\def\dispqed{\rlap{\qquad\qedsymbol}}

\opn\dis{dis}
\def\pnt{{\raise0.5mm\hbox{\large\bf.}}}

\opn\Lex{Lex}

\begin{document}
\title{\bf Depth and minimal number of generators of square free  monomial  ideals}

\author{ Dorin Popescu }

\thanks{The  support from  the CNCSIS grant PN II-542/2009 of Romanian Ministry of Education, Research and Inovation is gratefully acknowledged.}

\address{Dorin Popescu, Institute of Mathematics "Simion Stoilow", Research unit 5,
University of Bucharest, P.O.Box 1-764, Bucharest 014700, Romania}
\email{dorin.popescu@imar.ro}

\maketitle
\begin{abstract}
Let $I$ be an ideal of a polynomial algebra  over a field generated by square free monomials of degree $\geq d$. If $I$ contains more monomials of degree $d$ than
$(n-d)/(n-d+1)$ multiplied with the  number of square free monomials of $S$ of degree $d$ then $\depth_SI\leq d$, in particular the Stanley's Conjecture holds in this case.
  \vskip 0.4 true cm
 \noindent
  {\it Key words } : Monomial Ideals,  Depth, Stanley depth.\\
 {\it 2000 Mathematics Subject Classification: Primary 13C15, Secondary 13F20, 13F55,
13P10.}
\end{abstract}

Let $S=K[x_1,\ldots,x_n]$ be the polynomial algebra in $n$-variables over a field $K$  and $I\subset S$ a square free monomial ideal. Let $d$ be a positive integer and $\rho_d(I)$ be the number of all square free monomials of degree $d$ of $I$.

The  proposition below was repaired using an idea of Y. Shen to whom we owe thanks.
\begin{Proposition} \label{nice}  If $I$ is generated by square free monomials of degree $\geq d$ and \\
$\rho_d(I)> ((n-d)/(n-d+1)){n\choose d}$ then $\depth_SI\leq d$.
\end{Proposition}

\begin{proof}
Apply induction on $n$. If $n=d$ then there exists nothing to show. Suppose that $n>d$. Let $\nu_i$ be the number of the square free monomials of degree $d$ from $I\cap (x_i)$. We may consider two cases renumbering the variables if necessary.

{\bf Case 1}\ \  $\nu_1>((n-d)/(n-d+1)){n-1\choose d-1}$.

Let $S':=K[x_2,\ldots,x_n]$ and $x_1c_1,\ldots,x_1c_{\nu_1}$, $c_i\in S'$  be the square free monomials of degree $d$ from $I\cap (x_1)$. Then $J=(I:x_1)\cap S'$ contains $(c_1,\ldots,c_{\nu_1})$ and so $ \rho_{d-1}(J)\geq \nu_1>((n-d)/(n-d+1)){n-1\choose d-1}.$ By induction hypothesis, we get $\depth_{S'}J\leq d-1$. It follows $\depth_SJS\leq d$  and so $\depth_SI\leq d$ by \cite[Proposition 1.2]{R}.

{\bf Case 2}\ \ $\nu_i\leq ((n-d)/(n-d+1)){n-1\choose d-1}$ for all $i\in [n]$.

We get $\sum_{i=1}^n\nu_i\leq n((n-d)/(n-d+1)){n-1\choose d-1}$. Let $A_i$ be the set of the square free monomials of degree $d$ from $I\cap (x_i)$. A square free monomial from $I$ of degree $d$  will be present in $d$-sets $A_i$ and it follows
$$\rho_d(I)=|\cup_{i=1}^n A_i|\leq (n/d)((n-d)/(n-d+1)){n-1\choose d-1}= ((n-d)/(n-d+1)){n\choose d}$$
 if $n\geq d+1$. Contradiction!
\end{proof}

\begin{Remark}\label{h}{\em  If $I$ is generated by  square free monomials of  degree $\geq d$, then  $\depth_SI\geq d$.
Indeed, since $I$ has a square free resolution the last shift in the resolution of $I$ is at most $n$. Thus if $I$ is generated in degree $\geq d$, then the  resolution can have length at most $n-d$, which means that the depth of $I$ is greater than or equal to $d$ (this argument belongs to J. Herzog). Hence in the setting of the above proposition we get  $\depth_SI= d$.}
\end{Remark}

\begin{Corollary}\label{small} Let $I$ be an ideal generated by $\mu(I)$ square free monomials of degree $d$. If $\mu(I)> ((n-d)/(n-d+1)){n\choose d}$ then $\depth_SI= d$.
\end{Corollary}
\begin{Example} \label{1} {\em Let $I=(x_1x_2,x_2x_3)\subset S:=K[x_1,x_2,x_3].$ Then $d=2$ and $\mu(I)=2>(1/2){3\choose 2}$. It follows that $\depth_SI= 2$ by the above corollary.}
\end{Example}

\begin{Example}\label{2}{\em Let $I=(x_1x_2,x_1x_3,x_1x_4, x_2x_3,x_2x_5,x_3x_4,x_3x_5,x_4x_5)\subset \\
S:=K[x_1,\ldots,x_5]$. Then $d=2$ and $\mu(I)=8>(3/4){5\choose 2}$ and so $\depth_SI=2$.}
\end{Example}

Next lemma  presents a nice class of square free monomial ideals  $I$ with $\mu(I)={n\choose d+1}\leq ((n-d)/(n-d+1)){n\choose d}$ but  $\depth_SI= d$.
We suppose that $n\geq 3$.
Let  $w$ be the only square free monomial of degree $n$ of $S$, that is  $w=\Pi_{j=1}^n x_i$. Set $f_i=w/(x_ix_{i+1})$ for $1\leq i<n$, $f_n=w/(x_1x_n)$ and let $L_n:=(f_1,\ldots,f_{n-1})$, $I_n:=(L,f_n)$ be ideals of $S$ generated in degree $d=n-2$.  We will see that $\depth_SI_n=n-2$ even $\mu(I_n)=n={n\choose d+1}$.
\begin{Lemma}\label{n} Then $\depth_SL_n=n-1$ and $\depth_S I_n=n-2$.
\end{Lemma}
\begin{proof} Apply induction on $n\geq 3$. If $n=3$ then $L_3=(x_3,x_1)$, $I_3=(x_1,x_2,x_3)$ and the result is trivial. Assume that $n>3$. Note that $(L_n:x_n)=L_{n-1}S=(I_n:x_n)$ because $f_n,f_{n-1}\in (L_n:x_n)$. We have $$L_n=(L_n:x_n)\cap (x_n,L_n)=(L_{n-1}S)\cap (x_n,f_{n-1}),$$
$$I_n=(I_n:x_n)\cap (x_n,I_n)=(L_{n-1}S)\cap (x_n,f_{n-1},f_n)=(L_{n-1}S)\cap (x_n,u)\cap (x_1,x_{n-1},x_n),$$
where $u=w/(x_1x_{n-1}x_n )$. But $(x_1,x_{n-1})$ is a minimal prime ideal of  $L_{n-1}S$ and so we may remove $ (x_1,x_{n-1},x_n)$ above, that is
$I_n=(L_{n-1}S)\cap (x_n,u)$. On the other hand, $(L_{n-1}S)+ (x_n,u)=(x_n,I_{n-1})$ and $(L_{n-1}S)+ (x_n,f_{n-1})=(x_n,L_{n-1})S$ because $f_{n-1}\in L_{n-1}S$.
We have the following exact sequences
$$0\to S/L_{n}\to S/L_{n-1}S\oplus S/(x_n,f_{n-1})\to S/(x_n,L_{n-1}S)\to 0,$$
$$0\to S/I_{n}\to S/L_{n-1}S\oplus S/(x_n,u)\to S/(x_n,I_{n-1}S)\to 0.$$
By induction hypothesis $\depth L_{n-1}=n-2$ and  $\depth I_{n-1}=n-3$
and so $\depth_SS/(x_n,L_{n-1}S)=n-3$, $\depth_SS/(x_n,I_{n-1}S)=n-4$. As $\depth_S S/(x_n,f_{n-1})=\depth_S S/(x_n,u)=n-2$, it follows
$\depth_SS/L_n=n-2$, $\depth_S S/I_n=n-3$ by the Depth Lemma applied to the above exact sequences.
\end{proof}

Now, let $I$ be an arbitrary square free monomial ideal  and $P_I$ the poset given by all square free monomials of $I$ (a finite set) with the order given by the divisibility. Let ${\mathcal P}$ be a partition of $P_I$ in intervals $[u,v]=\{w\in P_I: u|w, w|v\}$, let us say  $P_I=\cup_i [u_i,v_i]$, the union being disjoint.
Define $\sdepth {\mathcal P}=\min_i\deg v_i$ and $\sdepth_SI=\max_{\mathcal P} \sdepth {\mathcal P}$, where ${\mathcal P}$ runs in the set of all partitions of $P_I$.
This is the so called the Stanley depth of $I$, in fact this is an equivalent definition given in a general form by \cite{HVZ}.

For instance, in Example \ref{1}, we have
$P_I=\{x_1x_2,x_2x_3,x_1x_2x_3\}$ and we may take ${\mathcal P}:\ \
P_I=[x_1x_2,x_1x_2x_3]\cup [x_2x_3,x_2x_3]$ with
$\sdepth_S{\mathcal P}=2$. Moreover, it is clear that   $\sdepth_SI=2$.
\begin{Remark} \label{st}{\em
  If $I$ is generated by $\mu(I)>{n\choose d+1}$ square free monomials of degree $d$ then $\sdepth_SI=d$.  Since $((n-d)/(n-d+1)){n\choose d}\geq {n\choose d+1}$, the Proposition \ref{nice} says that in a weaker case case $\depth_SI\leq \sdepth_SI$, which was in general  conjectured by Stanley \cite{S}. ´Stanley's Conjecture holds for intersections of four monomial prime ideals  of $S$ by \cite{AP} and \cite{P1} and for square free monomial ideals of $K[x_1,\ldots,x_5]$ by \cite{P} (a short exposition on this subject  is given in \cite{P2}). It is worth to mention that Proposition \ref{nice} holds in the stronger case when  $\mu(I)>{n\choose d+1}$ (see \cite{P3}), but the proof is much more complicated and the easy proof given in the present case has its importance.}
  \end{Remark}

In the Example \ref{2} we have
$P_I=[x_1x_2,x_1x_2x_4]\cup [x_1x_3,x_1x_3x_5]\cup [x_1x_4,x_1x_4x_5]\cup [x_2x_3,x_1x_2x_3]\cup [x_3x_4,x_1x_3x_4]\cup [x_3x_5,x_3x_4x_5]
\cup [x_4x_5,x_2x_4x_5]\cup [x_2x_3x_4,x_2x_3x_4]\cap [x_2x_3x_5,x_2x_3x_5]\cup (\cup_{\alpha} [\alpha,\alpha])$, where $\alpha$ runs in the set of square free monomials of $I$ of degree $4,5$.
It follows that $\sdepth_SI=3$. But as we know $\depth_SI=2$.


\vskip 0.5 cm

\begin{thebibliography}{99}


\bibitem{HVZ} J.\ Herzog, M.\ Vladoiu, X.\ Zheng, {\em How to compute the Stanley depth of a monomial ideal,}  J.  Algebra, 322 (2009), 3151-3169.
\bibitem{AP} A.\ Popescu, {\em Special Stanley Decompositions}, Bull. Math. Soc. Sc. Math. Roumanie, 53(101), no 4 (2010), arXiv:AC/1008.3680.
\bibitem{P} D.\ Popescu,  {\em An inequality between depth and Stanley depth}, Bull. Math. Soc. Sc. Math. Roumanie 52(100), (2009), 377-382, arXiv:AC/0905.4597v2.
\bibitem{P1} D.\ Popescu, {\em Stanley conjecture on intersections of four monomial prime ideals}, arXiv.AC/1009.5646.
\bibitem{P2} D.\ Popescu, {\em Bounds of Stanley depth}, An. St.  Univ. Ovidius. Constanta, 19(2),(2011), 187-194.
\bibitem{P3} D.\ Popescu, {\em Depth of factors of square free monomial ideals}, Preprint, 2011.
\bibitem{R} A.\ Rauf, {\em Depth and Stanley depth of multigraded modules}, Comm.  Algebra, 38 (2010),773-784.
\bibitem{S} R.\ P.\ Stanley, {\em Linear Diophantine equations and local cohomology,} Invent. Math. 68 (1982) 175-193.
\end{thebibliography}
\end{document}